\documentclass[a4paper,11pt]{amsart}
\usepackage{amsmath,amssymb,amsfonts,amsthm,exscale,calc}
\usepackage[top=3.0cm, bottom=3.0cm, inner=3.0cm, outer=3.0cm, includefoot]{geometry}

\usepackage[latin1]{inputenc}
\usepackage[T1]{fontenc}
\usepackage{verbatim}

\usepackage{geometry}
\usepackage{amssymb}
\usepackage{amsmath}
\usepackage{graphicx}
\usepackage{amsthm}

\usepackage{color}

\usepackage{enumerate}

\newcommand{\D}{\Deg_{\max}}
\newcommand{\IR}{{\mathbb{R}}}

\newcommand{\IN}{{\mathbb{N}}}

\newcommand{\eChar}{\begin{enumerate}[(i)]}
\newcommand{\eCharR}{\begin{enumerate}[(a)]}
\newcommand{\eBr}{\begin{enumerate}[(1)]}

\newcommand{\Deg}{\operatorname{Deg}}

\newcommand{\diam}{\operatorname{diam}}
\newcommand{\eps}{\varepsilon}

\newcommand{\Abstract}

\date{\today}

\theoremstyle{plain}
\newtheorem{lemma}{Lemma}[section]
\newtheorem{theorem}[lemma]{Theorem}

\newtheorem{corollary}[lemma]{Corollary}

\theoremstyle{definition}

\newtheorem{rem}[lemma]{Remark}
\newtheorem{defn}[lemma]{Definition}

\numberwithin{equation}{section}

\begin{document}

\title
{
	Bakry-\'Emery curvature and diameter bounds on graphs
}

\author[Liu]{Shiping Liu}
\address{S. Liu, School of Mathematical Sciences, University of Science and Technology of China, Hefei 230026, Anhui Province, China}
\email{spliu@ustc.edu.cn}

\author[M\"unch]{Florentin M\"unch}
\address{F. M\"unch, Institut f\"ur Mathematik\\Universit{\"a}t Potsdam \\14476 Potsdam, Germany }\email{chmuench@uni-potsdam.de}

\author[Peyerimhoff]{Norbert Peyerimhoff}
\address{N. Peyerimhoff, Department of Mathematical Sciences, Durham University, Durham DH1 3LE, United Kingdom}
\email{norbert.peyerimhoff@durham.ac.uk}

\begin{abstract}
	We prove diameter bounds for graphs having a positive Ricci-curvature bound in the Bakry-\'Emery sense. Our first result using only curvature and maximal vertex degree is sharp in the case of hypercubes. The second result depends on an additional dimension bound, but is independent of the vertex degree. In particular, the second result is the first Bonnet-Myers type theorem for unbounded graph Laplacians. Moreover, our results improve diameter bounds from \cite{Fathi2015} and \cite{Horn2014} and solve a conjecture from \cite{Cushing2016}.
\end{abstract}
\maketitle
\section{Introduction}

The classical Bonnet-Myers theorem states that for a complete, connected Riemannian manifold with Ricci-curvature bounded from below by $K>0$, the diameter is bounded by the diameter of the sphere with the same dimension and Ricci-curvature $K$ (see \cite{Myers1941}).
Moreover by Cheng's Rigidy theorem (see \cite{Cheng1975}), sharpness is obtained if and only if the manifold is a sphere. Bakry and Ledoux \cite{BL1996} successfully established this theorem for abstract Markov generators which are diffusion and  satisfy Bakry-\'Emery-Ricci curvature \cite{BE1985} conditions.


Our aim is to give a simple proof of this theorem in a discrete setting.
Indeed, discrete space Markov generators are not diffusion and therefore, the theory of Bakry and Ledoux is not applicable.
However, for many discrete curvature notions there is already a Bonnet-Myers-type result established (for sectional curvature on planar graphs, see \cite{DeVos2007,Higuchi2001,Keller2014,Stome1976}, and for Ollivier-Ricci-curvature, see \cite{Ollivier2009}, and for Forman's discrete Ricci curvature, see \cite{Forman2003}).

In this article, we focus on Bakry-\'Emery-Ricci-curvature \cite{BE1985,Schmueckenschlaeger1999,LY2010,Klartag2015}. 
Under the assumption of finite measure and bounded vertex degree, a
discrete Bonnet-Myers theorem has been proven for
Bakry-\'Emery-curvature on discrete Markov-chains \cite{Fathi2015}.
Furthermore, a discrete Bonnet-Myers type theorem was
established under the $CDE'$-condition in
\cite{Horn2014}, whereby $CDE'$ is stronger than the $CD$ condition
\cite{Muench2015}.  We will prove diameter bounds under $CD$
conditions which give sharp results and improve diameter bounds from
\cite{Fathi2015} and \cite{Horn2014}.  Moreover, our results solve
Conjecture~8.1 from \cite{Cushing2016}.

In contrast to manifolds, we can upper bound the Laplacian by the gradient on graphs, that is, $(\Delta f)^2 \leq C \Gamma f$ for all functions $f$ on the vertices and a constant $C > 0$ depending only on the maximal vertex degree. This property will give us diameter bounds using the vertex degree instead of the dimension parameter in the curvature-dimension condition.

\subsection{Organization of the paper and main results}
In Subsection~1.2, we define Bakry-\'Emery-Ricci-curvature and discuss different distance and diameter notions.

In Section~2, we give two versions of diameter bounds.

The first result (Corollary~\ref{cor:Bonnet Myers CD(K,infty)}) is using $CD(K,\infty)$ and bounded vertex degree $\D$ and gives the sharp estimate
$$
\diam_d(G) \leq \frac {2\D}{K}.
$$
This solves Conjecture~8.1 from \cite{Cushing2016} claiming that for every graph satisfying $CD(K,\infty)$ with $K>0$, there should exist an upper diameter bound of the graph only depending on $K$ and $\D$.

The second result (Theorem~\ref{thm:Bonnet-Myers CD(K,n)}) works in a more general setting, in particular, no boundedness of the vertex degree is needed anymore. But instead, we will assume $CD(K,n)$ with finite $n$ to prove
$$
\diam_\rho  \leq \pi \sqrt{\frac {n}K}
$$
where $\rho$ is the resistance metric (see Definition~\ref{def:resistance metric}).

Finally in Subsection~2.3, we compare these diameter bounds to diameter bounds from \cite{Fathi2015} and \cite{Horn2014}. In comparison to \cite{Horn2014}, we need a weaker curvature assumption and obtain stronger estimates (see Remark~\ref{rem:Horn}). In comparison to \cite{Fathi2015}, we have an improvement by a factor of $2$ under the same curvature assumptions (see Remark~\ref{rem:fathi}).

\subsection{Setup and notations}

A triple $G=(V,w,m)$ is called a \emph{(weighted) graph} if $V$ is a countable set, if $w:V^2 \to [0,\infty)$ is symmetric and zero on the diagonal and if $m:V \to (0,\infty)$. We call $V$ the \emph{vertex set}, $w$ the \emph{edge weight} and $m$ the \emph{vertex measure}. We define the \emph{graph Laplacian} as a map $\IR^V \to \IR^V$ via
$\Delta f(x) := \frac 1 {m(x)} \sum_y w(x,y)(f(y) - f(x))$.
In the following, we only consider \emph{locally finite} graphs, i.e., for every $x \in V$ there are only finitely many $y \in V$ with $w(x,y) >0$.
We write $\Deg(x) := \frac{\sum_y w(x,y)}{m(x)}$ and $\D := \sup_x \Deg(x)$.

\begin{defn} [Bakry-\'Emery-curvature]
The \emph{Bakry-\'Emery-operators} are defined via
$$
2\Gamma(f,g) := \Delta (fg) - f\Delta g - g\Delta f
$$
and
$$
2\Gamma_2(f,g) := \Delta \Gamma(f,g) - \Gamma(f, \Delta g) - \Gamma(g,\Delta f).
$$
We write $\Gamma(f):= \Gamma(f,f)$ and $\Gamma_2(f):=\Gamma_2(f,f)$.

A graph $G$ is said to satisfy the \emph{curvature dimension inequality} $CD(K,n)$ for some $K\in \IR$ and $n\in (0,\infty]$ if for all $f$,
$$
\Gamma_2(f) \geq \frac 1 n (\Delta f)^2 + K \Gamma f.
$$
\end{defn}

Next, we define combinatorial and resistance metrics and diameters.

\begin{defn}[Combinatorial metric]\label{def:combinatorial metric}
	Let $G=(V,w,m)$ be a locally finite graph.
	We define the \emph{combinatorial metric} $d:V^2 \to [0,\infty)$ via
	$$
	d(x,y) := \min \{n : \mbox{ there exist } x=x_0,\ldots,x_n=y \mbox{ s.t. } w(x_i,x_{i-1})>0 \,\mbox{for all } i=1\ldots n\}
	$$
	and the \emph{combinatorial diameter} as
	$\diam_d(G) := \sup_{x,y \in V} d(x,y)$.
\end{defn}

\begin{defn}[Resistance metric]\label{def:resistance metric}
	Let $G=(V,w,m)$ be a locally finite graph.
	We define the \emph{resistance metric} $\rho:V^2 \to [0,\infty)$ via
	$$
	\rho(x,y) := \sup \{f(y) - f(x) : \left\| \Gamma f \right\|_\infty \leq 1\}
	$$
	and the \emph{resistance diameter} as
	$\diam_\rho(G) := \sup_{x,y \in V} \rho(x,y)$.
\end{defn}

In the case of bounded degree, there is a standard estimate between combinatorial and resistance metric.

\begin{lemma}(Combinatorial and resistance metric)\label{lem:d, rho}
	Let $G=(V,w,m)$ be a locally finite graph with $\D < \infty$. Then for all $x_0,y_0 \in V$,
	$$
	d(x_0,y_0) \leq \sqrt{\frac{\D}{2}}\rho(x_0,y_0).
	$$
\end{lemma}

\begin{proof}
	Let $f:V \to \IR$ be defined by $f(x) := d(x,x_0)\sqrt{\frac{2} {\D}}$. Then for all $x \in V$,
\begin{align*}
	0\leq \Gamma f(x) &= \frac 1 {2 m(x)} \sum_y w(x,y)(f(y) - f(x))^2 \\
                &\leq    \frac 1 {2 m(x)} \sum_y w(x,y) \frac 2{\D}\\
                &=\frac{\Deg(x)}{\D}\\
                &\leq 1.
\end{align*}
Hence,
\begin{align*}
\rho(x_0,y_0) \geq f(y_0) - f(x_0) = d(y_0,x_0)\sqrt{\frac{2} {\D}}.
\end{align*}
This directly implies the claim.
\end{proof}

\section{Bonnet-Myers via the Bakry-\'Emery curvature-dimension condition}

In the first subsection, we obtain sharp diameter bounds for $CD(K,\infty)$.
In the second subsection, we obtain diameter bounds for unbounded Laplacians for $CD(K,n)$. In the third subsection we show that our results improve the diameter bounds from \cite[Theorem~7.10]{Horn2014} and from \cite[Corollary~6.4]{Fathi2015}.

\subsection{Diameter bounds and $CD(K,\infty)$}
The key to prove diameter bounds from $CD(K,\infty)$ is the semigroup characterization of $CD(K,\infty)$ which is equivalent to
$$
\Gamma P_t f \leq e^{-2Kt} P_t \Gamma f.
$$
{Here, $P_t$ denotes the heat semigroup operator.}
For details, see e.g. \cite{Gong2015,Lin2015}.

\begin{theorem}[Distance bounds under $CD(K,\infty)$]\label{thm:distance bounds CD(K,infty)}
	Let $(V,w,m)$ be a connected graph satisfying $CD(K,\infty)$ and $\D<\infty$.
	Then for all $x_0,y_0 \in V$,
	$$
	\rho(x_0,y_0) \leq  \frac {\sqrt{2\Deg(x_0)} + \sqrt{2\Deg(y_0)}}K.
	$$
\end{theorem}

\begin{proof}
	By \cite[Theorem~3.1]{Lin2015} and $\D<\infty$, we have that $CD(K,\infty)$ is equivalent to
	\begin{equation} \label{eq:semigrpchar}
	\Gamma P_t f \leq  e^{-2Kt} P_t \Gamma f	
	\end{equation}
	for all bounded functions $f:V\to \IR$.
	Due to Cauchy-Schwarz, $(\Delta g(x))^2 \leq 2 {\Deg(x)} (\Gamma g)(x)$ for all $x\in V$ and all $g:V\to \IR$.
	We fix $x_0 , y_0 \in V$ and $\eps>0$.
	Then by definition of $\rho$, there is a function $f:V\to \IR$ s.t.
	$f(y_0)-f(x_0) > \rho(x_0,y_0) - \eps$ and $\Gamma f \leq 1$.
	W.l.o.g., we can assume that $f$ is bounded.
Putting everything together yields for all $x \in  V$,		
	\begin{align*}
	\left|\partial_t P_t f(x)  \right|^2
	= \left|\Delta P_t f(x)  \right|^2
	\leq 2 {\Deg(x)}  \Gamma P_t f(x)
	\leq 2 {\Deg(x)}  e^{-2Kt} P_t \Gamma f(x)
	\leq 2 \Deg(x) e^{-2Kt}.
	\end{align*}
	By taking the square root and integrating from $t=0$ to $\infty$,
	we obtain
	\begin{align*}
	|P_T f(x) - f(x)|
	\leq \int_0^{\infty} \left|\partial_t P_t f(x)  \right| dt
	\leq \int_0^{\infty} \sqrt{2\Deg(x)} e^{-Kt} dt
	= \frac {\sqrt{2\Deg(x)}} K{}
	\end{align*}
	for all $T>0$ and $x \in V$.
	Now, the triangle inequality yields
	\begin{align*}
	\rho(x_0,y_0) - \eps &\leq |f(x_0) - f(y_0)| \\
	&\leq \left|P_t f(x_0) - f(x_0) \right|
	+ \left|P_t f(x_0) - P_t f(y_0) \right|
	+ \left|P_t f(y_0) - f(y_0) \right| \\
	& \leq  \frac {\sqrt{2\Deg(x_0)} + \sqrt{2\Deg(y_0)}}K +  \left|P_t f(x_0) - P_t f(y_0) \right| \\
	&\stackrel{t \to \infty}{\longrightarrow} \frac {\sqrt{2\Deg(x_0)} + \sqrt{2\Deg(y_0)}}K
	\end{align*}
	where $\left|P_t f(x_0) - P_t f(y_0) \right| \rightarrow 0$, since the graph is connected and since $\Gamma P_t f \rightarrow 0$ as $t \to \infty$ because of \eqref{eq:semigrpchar}.
	Taking the limit $\eps \to 0$ finishes the proof.
\end{proof}
We now use the distance bound to obtain a bound on the combinatorial diameter.
\begin{corollary}[Diameter bounds under $CD(K,\infty)$]\label{cor:Bonnet Myers CD(K,infty)}
	Let $(V,w,m)$ be a connected graph satisfying $CD(K,\infty)$.
	Then,
	$$
	\diam_d(G) \leq \frac {2 \D}K.
	$$
\end{corollary}
\begin{proof}
	Due to Lemma~\ref{lem:d, rho} and the above theorem, we have for all $x_0, y_0$,
	\begin{align*}
	d(x_0,y_0) \leq \sqrt{\frac{\D}{2}} \rho(x_0,y_0)
	\leq \sqrt{\frac{\D}{2}} \frac {\sqrt{2\Deg(x_0)} + \sqrt{2\Deg(y_0)}}K \leq \frac{2\D}{K}.
	\end{align*}
	Thus, $
	\diam_d(G) \leq \frac {2 \D}K$ as claimed.
\end{proof}

Indeed, this diameter bound is sharp for the $n$-dimensional hypercube which has diameter $n$, curvature bound $K=2$ and vertex degree $\D =n$.

An interesting question is whether an analog of Cheng's rigidy theorem (see \cite{Cheng1975}) holds true. In particular, we ask whether hypercubes are the only graphs for which the above diameter bound is sharp.

\subsection{Diameter bounds and $CD(K,n)$}

We can also give diameter bounds for unbounded Laplacians. We need two ingredients to do so.
First, we have to replace the combinatorial metric by a resistance metric.
Second, we have to assume a finite dimension bound.
Furthermore, we will need completeness of the graph and non-degenerate vertex measure to obtain the semigroup characterization of $CD(K,n)$ (see \cite[Theorem~3.3]{Gong2015}, \cite[Theorem~1.1]{Hua2015}). For definitions of completeness of graphs and non-degenerate vertex measure, see Sections 1 and 2.1 of \cite{Hua2015} or \cite[Definition~2.9, Definition~2.13]{Gong2015}.

We start with an easy but useful consequence of Gong and Lin's semigroup characterization of $CD(K,\infty)$, see \cite[Theorem~3.3]{Gong2015}.

\begin{lemma}[Semigroup property of $CD(K,n)$]\label{lem:semigroup characterization CD(K,n)}
   Let $G=(V,w,m)$ be a complete graph with non-degenerate vertex measure.
   Suppose $G$ satisfies $CD(K,n)$. Then for all bounded $f:V \to \IR$ with bounded $\Gamma f$,
\begin{align}\label{eqn:CD(K,n) semigroup characterization}
   \Gamma P_t f \leq e^{-2Kt} P_t \Gamma f - \frac{1-e^{-2Kt}}{Kn}(\Delta P_t f)^2.
\end{align}
\end{lemma}

\begin{proof} We first assume that $f$ is compactly supported.
	By \cite[Theorem~3.3]{Gong2015}, we have
	\begin{align}\label{eq:Char CD(K,n) Gong Lin}
	\Gamma P_t f \leq e^{-2Kt} P_t \Gamma f - \frac 2 n \int_0^t e^{-2Ks} P_s (\Delta P_{t-s} f)^2 ds.
	\end{align}
	Jensen's inequality yields $P_s g^2 \geq (P_s g)^2$ for all $g$ and thus by $g:=\Delta P_{t-s} f$,
\begin{align}\label{eq:Jensen semigroup}
\frac 2 n \int_0^t e^{-2Ks} P_s (\Delta P_{t-s} f)^2  ds
\geq  \frac 2 n \int_0^t e^{-2Ks} (P_s \Delta P_{t-s} f)^2 ds
=\frac{1-e^{-2Kt}}{Kn}(\Delta P_t f)^2.
\end{align}
	Putting (\ref{eq:Char CD(K,n) Gong Lin}) and (\ref{eq:Jensen semigroup}) together yields the claim for compactly supported $f$.
	We now prove the claim for all bounded $f$ with bounded $\Gamma(f)$ by completeness and a density argument.
    Completeness implies that there are compactly supported $\left(\eta_k\right)_{k \in \IN}$ s.t.
    $\eta_k \to 1$ from below and $\Gamma \eta_k \leq 1$.
    Due to compact support, (\ref{eqn:CD(K,n) semigroup characterization}) holds for $\eta_k f$. Obviously since $\eta_k \to 1$ from below, we have
    $\Gamma P_t(\eta_k f) \to \Gamma P_t f$ and $ \Delta P_t(\eta_k f) \to \Delta P_t f$, pointwise for $k \to \infty$.
    It remains to show
    $$
    P_t \Gamma (\eta_k f) \to P_t \Gamma f,
    $$
    pointwise for $k \to \infty$.
    We observe
\begin{align*}
    \left[ (\eta_k f)(y) - (\eta_k f)(x)\right]^2 &= \left[\eta_k(y)(f(y)-f(x)) + f(x)(\eta_k(y) - \eta_k(x)) \right]^2 \\
    &\leq 2\left[\eta_k(y)(f(y)-f(x)) \right]^2 + 2 \left[ f(x)(\eta_k(y) - \eta_k(x)) \right]^2\\
    &\leq 2\left\| \eta_k \right\|_\infty^2 (f(y)-f(x))^2 + 2\|f\|_\infty^2 (\eta_k(y) - \eta_k(x))^2
\end{align*}
    and thus,
    $$
    \Gamma(\eta_k f) \leq 2\left\| \eta_k \right\|_\infty^2 \Gamma f + 2\|f\|_\infty^2 \Gamma \eta_k.
    $$
    This implies that $\Gamma(\eta_k f)$ is uniformly bounded in $k$ and since $\eta_k f \to f$ pointwise, we obtain
    $P_t \Gamma (\eta_k f) \to P_t \Gamma f$ as desired.
\end{proof}

With this semigroup property in hands, we now can prove diameter bounds.
We will use similar methods as in the proof of Theorem~\ref{thm:distance bounds CD(K,infty)}.

\begin{theorem}[Diameter bounds under $CD(K,n)$]\label{thm:Bonnet-Myers CD(K,n)}
	Let $G=(V,w,m)$ be a connected, complete graph with non-degenerate vertex measure.
	Suppose $G$ satisfies $CD(K,n)$ for some $K>0$ and $n<\infty$. Then,
	$$
	\diam_\rho(G) \leq \pi \sqrt{\frac n K}.
	$$
\end{theorem}

\begin{proof}
	Suppose the opposite. Then, there are $x,y \in V$ s.t. $\rho(x,y)> \pi \sqrt{\frac n K}$, and there is a function $f:V \to \IR$ s.t. $\Gamma f \leq 1$ and $f(y) - f(x) >  \pi \sqrt{\frac n K}$. W.l.o.g., we can assume that $f$ is bounded.
	By the semigroup property of $CD(K,n)$ from Lemma~\ref{lem:semigroup characterization CD(K,n)} and by ignoring the non-negative term $\Gamma P_t f$, we have
	\begin{align*}
	\frac{1-e^{-2Kt}}{Kn}\left( \Delta P_t f \right)^2 \leq e^{-2Kt} P_t \Gamma f.
	\end{align*}
    Taking square root and applying $\Gamma f \leq 1$ yields
    $$
    |\partial_t P_t f| = |\Delta P_t f| \leq \sqrt{Kn} \sqrt{\frac{e^{-2Kt}}{1-e^{-2Kt}}} = \sqrt{Kn} \sqrt{\frac{1}{e^{2Kt}-1}}.
    $$
	Integrating from $t=0$ to $\infty$ yields for all $T$,
    \begin{align*}
       |P_T f - f|
       \leq \sqrt {Kn} \int_{0}^{\infty} \sqrt{\frac{1}{e^{2Kt}-1}} dt
       =  \sqrt {Kn} \frac{ \arctan{\sqrt{ e^{2Kt} - 1  }}}{K} \bigg|_{t=0}^{\infty}  = \frac {\pi}{2} \sqrt{\frac{n}{K}}
    \end{align*}	

    Due to $CD(K,\infty)$ which implies $\| \Gamma P_t f \|_\infty \stackrel{t\to \infty}{\longrightarrow} 0$ and since $G$ is connected, we infer
    $\left|P_t f(x) - P_t f(y) \right| \stackrel{t\to \infty}{\longrightarrow} 0$.
    We now apply the triangle inequality and obtain
    \begin{align*}
    \pi \sqrt{\frac n K}
    &< |f(y) - f(x)| \\
    &\leq |P_t f(y) - f(y)|  + |P_t f(y) - P_t f(x)|  +|P_t f(x) - f(x)| \\
    &\leq \pi \sqrt{\frac n K} +  |P_t f(y) - P_t f(x)| \\
    &\stackrel{t\to \infty}{\longrightarrow}  \pi \sqrt{\frac n K}.
    \end{align*}
    This is a contradiction and thus, $\diam_\rho(G) \leq \pi \sqrt{\frac n K}$ as claimed.
\end{proof}

In contrast to Corollary~\ref{cor:Bonnet Myers CD(K,infty)}, we cannot have sharpness in Theorem~\ref{thm:Bonnet-Myers CD(K,n)} since in the proof, we have thrown away $\Gamma P_t f$ which is strictly positive for $t>0$.

\subsection{Comparison with other discrete diameter bounds}

In \cite[Theorem 7.10]{Horn2014}, Horn, Lin, Liu and Yau have proven
\[
CDE'(K,n) \quad \Longrightarrow \quad \diam_d(G) \leq 2\pi \sqrt{\frac{6 \D n}{K} }.
\]

Indeed, due to \cite[Corollary~3.3]{Muench2015} and Theorem~\ref{thm:Bonnet-Myers CD(K,n)}, this result can be improved to

\[
CDE'(K,n) \quad \Longrightarrow \quad CD(K,n) \Longrightarrow \quad     \diam_\rho(G) \leq \pi \sqrt{\frac{n}{K}}.
\]
In case of $\D < \infty$, we have
$$
\diam_d(G) \leq \sqrt{\frac{ \D}{2}} \diam_\rho(G) \leq \pi \sqrt{\frac{\D n}{ 2K}},
$$
where the first estimate is due to Lemma~\ref{lem:d, rho}.

\begin{rem}\label{rem:Horn}
	Summarizing, we can say that our approach improves \cite[Theorem 7.10]{Horn2014} by a factor of $4\sqrt 3$ and by having weaker curvature assumptions.
\end{rem}

Let us also compare our results to the results on Markov-chains in \cite{Fathi2015}.
Translated into the graph setting, we have the following result in \cite{Fathi2015}.

\begin{theorem}(see  \cite[Corollary~6.4]{Fathi2015})\label{thm:Fathi}
Let $G=(V,w,m)$ be a graph with $\D < \infty$ and $m(V) := \sum_x m(x) < \infty$.
Suppose $G$ satisfies $CD(K,\infty)$ for some $K>0$.
 Then for all $x,y \in V$,
\begin{align*}
\rho(x,y) \leq 2\sqrt 2 \frac{ \left(\sqrt{\Deg(x)} + \sqrt{\Deg(y)}\right)}{K}.
\end{align*}
\end{theorem}

\begin{proof}
	Indeed, the theorem is just a reformulation of \cite[Corollary~6.4]{Fathi2015}.
Let $G=(V,w,m)$ be a graph with $m(V) < \infty$ and $\D < \infty$. We define the corresponding Markov kernel as
$$K(x,y) := \begin{cases}
\frac{w(x,y)}{m(x)\D} &: x\neq y\\
1 - \frac{\Deg(x)}{\D}&: x=y
\end{cases} $$	
and the corresponding measure $\mu(x) := m(x)/m(V)$ for all $x\in V$ according to  \cite{Fathi2015}.
Then, $Lf(x) := \sum_y (f(y) - f(x)) K(x,y) = \frac {\Delta f(x)}{\D}$.

Since $\Delta$ satisfies $CD(K,\infty)$, we have that $L$ satisfies $CD(\frac K {\D} , \infty)$.
As in \cite{Fathi2015}, we set
$$
J(x) := 1- K(x,x) =  \frac{\Deg(x)}{\D}.
$$
For all $x,y\in V$, we have
\begin{align*}
d_\Gamma(x,y) := \sup\left\{f(y)-f(x): \frac 1 2 L(f^2) - fLf \leq 1\right\} = \sqrt{\D} \rho(x,y).
\end{align*}

Now, \cite[Corollary~6.4]{Fathi2015} yields for all $x,y \in V$,
\begin{align*}
 \rho(x,y) \frac{K}{\sqrt{\D}} =d_\Gamma(x,y)\frac{K}{\D} &\leq 2\sqrt 2 \left(\sqrt{J(x)} + \sqrt{J(y)} \right)  \\
&= \frac{2 \sqrt{2} \left(\sqrt{\Deg(x)} + \sqrt{\Deg(y)} \right)} {\sqrt{\D}}.
\end{align*}
Multiplying with $\frac{\sqrt{\D}}{K}$ finishes the proof.	
\end{proof}

\begin{rem}\label{rem:fathi}
We observe that Theorem~\ref{thm:distance bounds CD(K,infty)} improves Fathi's and Shu's result \cite[Corollary~6.4]{Fathi2015}, see Theorem~\ref{thm:Fathi}, by a factor of $2$	
and by the fact, that we allow $m(V) = \infty$ which corresponds to an infinite reversible invariant measure in the Markov-chain setting.	
\end{rem}

\medskip

{\bf Acknowledgement:} We gratefully acknowledge partial support by the
EPSRC Grant EP/K016687/1. FM wants to thank the German Research Foundation (DFG) for financial support.


\begin{thebibliography}{9}
\bibitem{BE1985} \textsc{D. Bakry, M. \'Emery}, \textit{Diffusions hypercontractives
  (French) [Hypercontractive diffusions]}, S\'eminaire de
  probabilit\'es, XIX, 1983/84, Lecture Notes in Math. 1123,
  J. Az\'{e}ma and M. Yor (Editors), Springer, Berlin,
  pp. 177-206, (1985).

\bibitem{BL1996}\textsc{D. Bakry, M. Ledoux}, \textit{Sobolev inequalities and Myers's diameter theorem for an abstract Markov generator},
Duke Math. J. 85 (1), 253-270, (1996).

\bibitem{Cheng1975}\textsc{S. Y. Cheng},\ \textit{Eigenvalue comparison theorems and its geometric applications}, Math. Z. 143 (3), 289-297, (1975).

\bibitem{Cushing2016}\textsc{D. Cushing, S. Liu, N. Peyerimhoff},\ \textit{Bakry-\'Emery curvature functions of graphs}, arXiv: 1606.01496, (2016).

\bibitem{DeVos2007}\textsc{M. DeVos, B. Mohar},\ \textit{An analogue of the Descartes-Euler formula for infinite
	graphs and Higuchi's conjecture}, Trans. Am. Math. Soc. 359 (7), 3287-3300, (2007).

\bibitem{Fathi2015}\textsc{M. Fathi, Y. Shu},\ \textit{Curvature and transport inequalities for Markov chains in discrete spaces}, arXiv: 1509.07160, (2015).

\bibitem{Forman2003}\textsc{R. Forman},\ \textit{Bochner's method for cell complexes and combinatorial Ricci curvature},
Discrete Comput. Geom. 29 (3), 323-374, (2003).

\bibitem{Gong2015}\textsc{C. Gong, Y. Lin},\ \textit{Properties for CD Inequalities with Unbounded Laplacians}, arXiv: 1512.02471, (2015).

\bibitem {Higuchi2001}\textsc{Y. Higuchi},\ \textit{Combinatorial Curvature for Planar Graphs},  J. Graph Theory 38 (4), 220-229, (2001).

\bibitem {Horn2014}\textsc{P. Horn, Y. Lin, Shuang Liu, S.-T. Yau},\ \textit{Volume doubling, Poincar\'e inequality and Guassian heat kernel estimate for nonnegative curvature graphs},  arXiv: 1411.5087v3, (2014).

\bibitem {Hua2015}\textsc{B. Hua, Y. Lin},\ \textit{Stochastic completeness for graphs with curvature dimension conditions},  arXiv: 1504.00080v2, (2015).

\bibitem {Klartag2015}\textsc{B. Klartag, G. Kozma, P. Ralli, P. Tetali},\ \textit{Discrete curvature and abelian groups},  Canad. J. Math. 68 (3), 655-674, (2016).

\bibitem {Keller2014}\textsc{M. Keller, N. Peyerimhoff, F. Pogorzelski},\ \textit{Sectional curvature of polygonal complexes with planar substructures},  arXiv: 1407.4024, (2014), to appear in Adv. Math. (2016), http://dx.doi.org/10.1016/j.aim.2016.10.027



\bibitem {Lin2015}\textsc{Y. Lin, Shuang Liu},\ \textit{Equivalent Properties of CD Inequality on Graph},  arXiv: 1512.02677, (2015).

\bibitem{LY2010} \textsc{Y. Lin, S.-T. Yau}, \ \textit{Ricci curvature and eigenvalue
  estimate on locally finite graphs}, Math. Res. Lett. 17 (2), 343-356, (2010).


\bibitem{Muench2015}\textsc{F. M\"unch},\ \textit{Remarks on curvature dimension conditions on graphs}, arXiv: 1501.05839, (2015).


\bibitem{Myers1941}\textsc{S. B. Myers},\ \textit{Riemannian manifolds with positive mean curvature}, Duke Math. J. 8, 401-404, (1941).

\bibitem{Ollivier2009}\textsc{Y. Ollivier},\ \textit{Ricci curvature of Markov chains on metric spaces}, J. Funct. Anal. 256 (3), 810-864, (2009).


\bibitem{Schmueckenschlaeger1999}\textsc{M. Schmuckenschl\"ager},\ \textit{Curvature of nonlocal Markov generators},  Convex geometric analysis (Berkeley, CA, 1996), volume 34 of Math. Sci. Res. Inst. Publ., 189-197, Cambridge Univ. Press, Cambridge, 1999.

\bibitem {Stome1976}\textsc{D. A. Stone},\ \textit{A combinatorial analogue of a theorem of Myers},  Illinois J. Math. 20 (1), 12-21, (1976) and Erratum: Illinois J. Math. 20 (3), 551-554, (1976).








\end{thebibliography}
\end{document}